\documentclass{article}

\usepackage{amsmath}
\usepackage{amsthm}

\usepackage[pdftex]{hyperref}
\usepackage{alltt}

\newtheorem{theorem}{Theorem}[section]

\numberwithin{equation}{section}

\title{Some Harmonic Number Identities\\ involving certain Reciprocals}
\author{M.J. Kronenburg}

\begin{document}

\maketitle

\begin{abstract}
Some finite series of harmonic numbers involving certain reciprocals are evaluated.
Products of such reciprocals are expanded in a sum of the individual reciprocals,
leading to a computer program. A list of examples is provided.
\end{abstract}

\noindent
\textbf{Keywords}: harmonic number.\\
\textbf{MSC 2010}: 11B99

\section{Definitions and Basic Identities}

The generalized harmonic numbers used in this paper are:
\begin{equation}
  H_n^{(m)} = \sum_{k=1}^n \frac{1}{k^m}
\end{equation}
from which follows that $H_0^{(m)}=0$.
The traditional harmonic numbers are:
\begin{equation}
 H_n=H_n^{(1)}
\end{equation}
A well known identity is \cite{GKP94,K97,SLL89,S90}:
\begin{equation}\label{sum1okhk}
 \sum_{k=1}^n \frac{1}{k} H_k = \frac{1}{2} ( H_n^2 + H_n^{(2)} )
\end{equation}
and \cite{K97,K11a}:
\begin{equation}\label{sum1okp1hk}
 \sum_{k=0}^n \frac{1}{k+1}H_k = \frac{1}{2} ( H_{n+1}^2 - H_{n+1}^{(2)} )
\end{equation}
and \cite{K11b}:
\begin{equation}
 \sum_{k=1}^n\frac{1}{k}H_{n-k} = H_n^2 - H_n^{(2)}
\end{equation}
\begin{equation}\label{sum1okp1hnmk}
 \sum_{k=0}^n \frac{1}{k+1} H_{n-k} = H_{n+1}^2 - H_{n+1}^{(2)}
\end{equation}

When $a\leq b$ are two integers and $\{x_k\}$ and $\{y_k\}$ are two sequences of complex numbers,
and $\{s_k\}$ the sequence of complex numbers defined by:
\begin{equation}
 s_k = \sum_{i=a}^k x_i
\end{equation}
then there is the following summation by parts formula \cite{K11a}:
\begin{equation}\label{sumparts}
 \sum_{k=a}^{b-1} x_k y_k = s_{b-1} y_b - \sum_{k=a}^{b-1} s_k ( y_{k+1} - y_k )
\end{equation}

\section{Harmonic Number Identities with a Reciprocal}

\begin{theorem}
For nonnegative integer $n$ and integer $p>0$:
\begin{equation}\label{hrec1}
\begin{split}
 \sum_{k=0}^n \frac{1}{k+p}H_k & = H_{n+p}(H_{n+1}+H_{p-1}) 
  - \frac{1}{2} [ (H_{n+1}+H_{p-1})^2 + H_{n+1}^{(2)} + H_{p-1}^{(2)} ] \\
 & - \sum_{k=0}^{p-2} \frac{1}{n+k+2} H_k
\end{split}
\end{equation}
\end{theorem}
\begin{proof}
Summation by parts (\ref{sumparts}) with $x_k=1/(k+p)$ and $y_k=H_k$ yields:
\begin{equation}
 \sum_{k=1}^n \frac{1}{k+p} H_k = (H_{n+p}-H_p)H_{n+1} - 
   \sum_{k=1}^n \frac{1}{k+1} ( H_{k+p} - H_p ) 
\end{equation}
Using:
\begin{equation}
 H_{k+p} = H_k + \sum_{s=1}^p \frac{1}{k+s}
\end{equation}
and for $s>1$:
\begin{equation}
 \frac{1}{(k+s)(k+1)} = \frac{1}{s-1} ( \frac{1}{k+1} - \frac{1}{k+s} )
\end{equation}
yields:
\begin{equation}
 \frac{1}{k+1} H_{k+p} = \frac{1}{k+1} H_k + \frac{1}{(k+1)^2} 
  + \sum_{s=2}^p \frac{1}{s-1} ( \frac{1}{k+1} - \frac{1}{k+s} )
\end{equation}
Performing the summation over $n$ and using (\ref{sum1okp1hk}) yields:
\begin{equation}
\begin{split}
 \sum_{k=1}^n \frac{1}{k+p}H_k & = H_{n+p}H_{n+1} - \frac{1}{2} ( H_{n+1}^2 + H_{n+1}^{(2)} )
  - H_p + 1 \\
  & + \sum_{s=1}^{p-1} \frac{1}{s} ( H_{n+s+1} - H_{n+1} - H_{s+1} + 1 )
\end{split}
\end{equation}
Using
\begin{equation}
 H_{n+s+1} - H_{n+1} = \sum_{k=1}^s \frac{1}{n+k+1}
\end{equation}
and changing the order of summation over $s$ and $k$:
\begin{equation}
\begin{split}
 \sum_{s=1}^{p-1} \frac{1}{s} \sum_{k=1}^s \frac{1}{n+k+1} & =
  \sum_{k=1}^{p-1} \frac{1}{n+k+1} \sum_{s=k}^{p-1} \frac{1}{s} \\
 & = \sum_{k=1}^{p-1} \frac{1}{n+k+1} ( H_{p-1} - H_{k-1} ) \\
 & = H_{p-1} ( H_{n+p} - H_{n+1} ) - \sum_{k=0}^{p-2} \frac{1}{n+k+2} H_k
\end{split}
\end{equation}
and using $H_{s+1}=H_s+1/(s+1)$
and (\ref{sum1okhk}) and $1/(s(s+1))=1/s-1/(s+1)$ yields the theorem.
\end{proof}

\begin{theorem}
For nonnegative integer $n$ and integer $p>0$:
\begin{equation}\label{hrec2}
\begin{split}
 \sum_{k=0}^n \frac{1}{k+p}H_{n-k} & = H_{n+1} ( H_{n+p} - H_{p-1} )
  - \frac{1}{2}[ H_{n+1}^2 - H_{n+p}^2 + H_{n+1}^{(2)} + H_{n+p}^{(2)} ] \\
 & - \sum_{k=0}^{p-2} \frac{1}{n+k+2} H_k
\end{split}
\end{equation}
\end{theorem}
\begin{proof}
Summation by parts (\ref{sumparts}) with $x_k=1/(n+p-k)$ and $y_k=H_k$ yields:
\begin{equation}
\begin{split}
 \sum_{k=0}^n \frac{1}{k+p} H_{n-k} & = \sum_{k=1}^n \frac{1}{n+p-k} H_k \\
 & = (H_{n+p-1}-H_{p-1})H_{n+1} - \sum_{k=1}^n \frac{1}{k+1} ( H_{n+p-1} - H_{n+p-k-1} ) 
\end{split}
\end{equation}
Using:
\begin{equation}
 H_{n+p-k-1} = H_{n-k} + \sum_{s=1}^{p-1} \frac{1}{n+s-k}
\end{equation}
and for $s>0$:
\begin{equation}
 \frac{1}{(n+s-k)(k+1)} = \frac{1}{n+s+1} ( \frac{1}{k+1} + \frac{1}{n+s-k} )
\end{equation}
yields:
\begin{equation}
 \frac{1}{k+1} H_{n+p-k-1} = \frac{1}{k+1} H_{n-k} 
  + \sum_{s=1}^{p-1} \frac{1}{n+s+1} ( \frac{1}{k+1} + \frac{1}{n+s-k} )
\end{equation}
Performing the summation over $n$ and using (\ref{sum1okp1hnmk}) yields:
\begin{equation}
\begin{split}
 \sum_{k=1}^n \frac{1}{k+p}H_{n-k} & = H_{n+p-1} - H_{p-1}H_{n+1} + H_{n+1}^2 - H_{n+1}^{(2)} - H_n \\
  & + \sum_{s=1}^{p-1} \frac{1}{n+s+1} ( H_{n+s-1} + H_{n+1} - H_{s-1} - 1 )
\end{split}
\end{equation}
Using
\begin{equation}
 \sum_{s=1}^{p-1} \frac{1}{n+s+1}(H_{n+1}-1) = (H_{n+1}-1) (H_{n+p} - H_{n+1})
\end{equation}
and $H_{n+s-1}=H_{n+s}-1/(n+s)$ and $1/((n+s)(n+s+1))=1/(n+s)-1/(n+s+1)$ and with (\ref{sum1okp1hnmk}):
\begin{equation}
\begin{split}
\sum_{s=0}^{p-2} \frac{1}{n+s+2}H_{n+s+1} & = \sum_{s=0}^{n+p-1}\frac{1}{k+1}H_k - \sum_{s=0}^n\frac{1}{k+1}H_k \\
 & = \frac{1}{2}( H_{n+p}^2 - H_{n+p}^{(2)} - H_{n+1}^2 + H_{n+1}^{(2)} )
\end{split}
\end{equation}
yields the theorem.
\end{proof}

\begin{theorem}
For nonnegative integer $n$ and integer $0\leq p\leq n$:
\begin{equation}
\begin{split}
 \sum_{k=p+1}^n \frac{1}{k-p}H_k & =  
  \frac{1}{2} [ (H_{n-p+1}+H_p)^2 + H_{n-p+1}^{(2)} + H_p^{(2)} ] 
  - H_{n+1} ( H_p + \frac{1}{n-p+1} ) \\
 & + \sum_{k=0}^{p-1} \frac{1}{n-p+k+2} H_k
\end{split}
\end{equation}
\end{theorem}
\begin{proof}
Summation by parts (\ref{sumparts}) with $x_k=1/(k-p)$ and $y_k=H_k$ yields:
\begin{equation}
\begin{split}
 \sum_{k=p+1}^n \frac{1}{k-p} H_k & = H_{n+1} H_{n-p} - 
   \sum_{k=p+1}^n \frac{1}{k+1} H_{k-p} \\
 & = H_{n+1} H_{n-p} - \sum_{k=1}^{n-p} \frac{1}{k+p+1} H_k
\end{split}
\end{equation}
The last sum is (\ref{hrec1}) with $p$ replaced by $p+1$ and $n$ by $n-p$,
which yields the theorem.
\end{proof}

\begin{theorem}
For nonnegative integer $n$ and integer $0\leq p\leq n$:
\begin{equation}
 \sum_{k=p+1}^n \frac{1}{k-p}H_{n-k} =  H_{n-p}^2 - H_{n-p}^{(2)}
\end{equation}
\end{theorem}
\begin{proof}
\begin{equation}
 \sum_{k=p+1}^n \frac{1}{k-p}H_{n-k} = \sum_{k=0}^{n-p-1} \frac{1}{k+1} H_{n-p-k-1}
\end{equation}
The last sum is (\ref{sum1okp1hnmk}) with $n$ replaced by $n-p-1$, which yields the theorem.
\end{proof}

\section{Products of Reciprocals}

A finite product of these reciprocals with different $p$'s can be written as
a sum of the individual reciprocals.
The formula for two reciprocals is, where $p_1\neq p_2$:
\begin{equation}
\begin{split}
 \frac{1}{(k+p_1)(k+p_2)} & = \frac{1}{p_1-p_2}\frac{(k+p_1)-(k+p_2)}{(k+p_1)(k+p_2)} \\
 & = \frac{1}{p_1-p_2}( \frac{1}{k+p_2} - \frac{1}{k+p_1} )
\end{split}
\end{equation}
The formula for three reciprocals is, where $p_1\neq p_2\neq p_3$:
\begin{equation}
\begin{split}
 \frac{1}{(k+p_1)(k+p_2)(k+p_3)} & = \frac{1}{p_1-p_2}[ ( \frac{1}{p_2-p_3} - \frac{1}{p_1-p_3} ) \frac{1}{k+p_3} \\
 &  - \frac{1}{p_2-p_3} \frac{1}{k+p_2} + \frac{1}{p_1-p_3} \frac{1}{k+p_1} ]
\end{split}
\end{equation}
The recursion formula for $m$ reciprocals in terms of the formula for $m-1$ reciprocals is:
\begin{equation}
 \prod_{i=1}^{m-1} \frac{1}{k+p_i} = \sum_{i=1}^{m-1} \alpha_i \frac{1}{k+p_i}
\end{equation}
\begin{equation}
\begin{split}
 \prod_{i=1}^m \frac{1}{k+p_i} & = \sum_{i=1}^{m-1} \alpha_i \frac{1}{k+p_m}\frac{1}{k+p_i} \\
 & =  - \sum_{i=1}^{m-1} \alpha_i \frac{1}{p_i-p_m} \frac{1}{k+p_i} \\
 & + \frac{1}{k+p_m} \sum_{i=1}^{m-1} \alpha_i \frac{1}{p_i-p_m}
\end{split}
\end{equation}
This recursion formula means that starting with $m=1$ and $\alpha_1=1$,
in each pass for certain $m>1$ the $\alpha_i$ for $i=1\cdots m-1$ are divided by $p_m-p_i$,
after which $\alpha_m$ is minus the sum of the new $\alpha_i$ for $i=1\cdots m-1$.
This way the recursion formula reduces to a double iteration, and it is also clear from this
that for $m>1$:
\begin{equation}\label{alphazero}
 \sum_{i=1}^m \alpha_i = 0
\end{equation}
When the $\alpha_i$ have been computed, each individual reciprocal can be summed using the
appropriate formula in the previous section, where the following substitutions are made:
\begin{equation}
 H_{n+p}^{(m)} = H_{n+1}^{(m)} + \sum_{k=2}^p \frac{1}{(n+k)^m}
\end{equation}
\begin{equation}
 H_{n-p+1}^{(m)} = H_{n+1}^{(m)} - \sum_{k=0}^{p-1} \frac{1}{(n-k+1)^m}
\end{equation}
After these substitutions the coefficient of $H_{n+1}^2$ in the formula for each individual reciprocal
is identical, and therefore, by equation (\ref{alphazero}), the coefficient of $H_{n+1}^2$ in the resulting
formula for $m>1$ is zero, which means that for these products of these reciprocals only terms
linear in harmonic numbers remain.

\section{Examples}

\begin{equation}
 \sum_{k=0}^n \frac{1}{k+1} H_k = \frac{1}{2} ( H_{n+1}^2 - H_{n+1}^{(2)} )
\end{equation}
\begin{equation}
 \sum_{k=0}^n \frac{1}{k+2} H_k = \frac{1}{2} ( H_{n+1}^2 - H_{n+1}^{(2)} ) + \frac{1}{n+2}H_{n+1} - \frac{n+1}{n+2}
\end{equation}
\begin{equation}
 \sum_{k=0}^n \frac{1}{k+3} H_k = \frac{1}{2} ( H_{n+1}^2 - H_{n+1}^{(2)} ) + \frac{2n+5}{(n+2)(n+3)}H_{n+1} - \frac{(n+1)(7n+20)}{4(n+2)(n+3)}
\end{equation}
\begin{equation}
 \sum_{k=1}^n \frac{1}{k} H_k = \frac{1}{2} ( H_n^2 + H_n^{(2)} )
\end{equation}
\begin{equation}
 \sum_{k=2}^n \frac{1}{k-1} H_k = \frac{1}{2} ( H_{n+1}^2 + H_{n+1}^{(2)} ) - \frac{2n+1}{n(n+1)}H_{n+1} + \frac{n}{n+1}
\end{equation}
\begin{equation}
 \sum_{k=3}^n \frac{1}{k-2} H_k = \frac{1}{2} ( H_{n+1}^2 + H_{n+1}^{(2)} ) - \frac{3n^2-1}{(n-1)n(n+1)}H_{n+1} + \frac{7n^2-n-2}{4n(n+1)}
\end{equation}
\begin{equation}
 \sum_{k=1}^n \frac{1}{k(k+1)}H_k = H_{n+1}^{(2)} - \frac{1}{n+1} H_{n+1}
\end{equation}
\begin{equation}
 \sum_{k=0}^n \frac{1}{(k+1)(k+2)}H_k = \frac{n+1}{n+2} - \frac{1}{n+2} H_{n+1}
\end{equation}
\begin{equation}
 \sum_{k=2}^n \frac{1}{k(k-1)}H_k = \frac{2n+1}{n+1} - \frac{1}{n} H_{n+1}
\end{equation}
\begin{equation}
 \sum_{k=3}^n \frac{1}{(k-1)(k-2)}H_k = \frac{9n^2+5n-2}{4n(n+1)} - \frac{1}{n-1}H_{n+1}
\end{equation}
\begin{equation}
 \sum_{k=1}^n \frac{1}{k(k+1)(k+2)}H_k = \frac{1}{2} H_{n+1}^{(2)} - \frac{1}{2(n+1)(n+2)} H_{n+1} - \frac{n+1}{2(n+2)}
\end{equation}
\begin{equation}
 \sum_{k=2}^n \frac{1}{(k+1)k(k-1)}H_k = \frac{5n+3}{4(n+1)} - \frac{1}{2n(n+1)} H_{n+1} - \frac{1}{2} H_{n+1}^{(2)}
\end{equation}
\begin{equation}
 \sum_{k=3}^n \frac{1}{k(k-1)(k-2)}H_k = \frac{2n^2+2n-1}{4n(n+1)} - \frac{1}{2n(n-1)} H_{n+1}
\end{equation}
\begin{equation}
\begin{split}
 \sum_{k=2}^n \frac{1}{(k+2)(k+1)k(k-1)}H_k & = \frac{23n^2+57n+28}{36(n+1)(n+2)} \\
 & - \frac{1}{3n(n+1)(n+2)} H_{n+1} - \frac{1}{3} H_{n+1}^{(2)}
\end{split}
\end{equation}
\begin{equation}
\begin{split}
 \sum_{k=3}^n \frac{1}{(k+1)k(k-1)(k-2)}H_k & = \frac{1}{6} H_{n+1}^{(2)} - \frac{1}{3(n-1)n(n+1)} H_{n+1} \\
 & - \frac{2n^2+1}{12n(n+1)} \\
 &  
\end{split}
\end{equation}
\begin{equation}
 \sum_{k=0}^n \frac{1}{k+1} H_{n-k} =  H_{n+1}^2 - H_{n+1}^{(2)}
\end{equation}
\begin{equation}
 \sum_{k=0}^n \frac{1}{k+2} H_{n-k} =  H_{n+1}^2 - H_{n+1}^{(2)} - \frac{n}{n+2} H_{n+1}
\end{equation}
\begin{equation}
 \sum_{k=0}^n \frac{1}{k+3} H_{n-k} =  H_{n+1}^2 - H_{n+1}^{(2)} - \frac{3n^2+7n-2}{2(n+2)(n+3)} H_{n+1} - \frac{n+1}{(n+2)(n+3)}
\end{equation}
\begin{equation}
 \sum_{k=1}^n \frac{1}{k} H_{n-k} =  H_n^2 - H_n^{(2)}
\end{equation}
\begin{equation}
 \sum_{k=2}^n \frac{1}{k-1} H_{n-k} =  H_{n+1}^2 - H_{n+1}^{(2)} - \frac{2(2n+1)}{n(n+1)} H_{n+1} + \frac{2(3n^2+3n+1)}{n^2(n+1)^2}
\end{equation}
\begin{equation}
 \sum_{k=3}^n \frac{1}{k-2} H_{n-k} =  H_{n+1}^2 - H_{n+1}^{(2)} - \frac{2(3n^2-1)}{(n-1)n(n+1)} H_{n+1} + \frac{2(6n^4-3n^2+1)}{(n-1)^2n^2(n+1)^2}
\end{equation}
\begin{equation}
 \sum_{k=1}^n \frac{1}{k(k+1)}H_{n-k} = \frac{n-1}{n+1} H_{n+1} - \frac{n-1}{(n+1)^2} 
\end{equation}
\begin{equation}
 \sum_{k=0}^n \frac{1}{(k+1)(k+2)}H_{n-k} = \frac{n}{n+2} H_{n+1} 
\end{equation}
\begin{equation}
 \sum_{k=2}^n \frac{1}{k(k-1)}H_{n-k} = \frac{n-2}{n} H_{n+1} - \frac{(n-2)(2n+1)}{n^2(n+1)} 
\end{equation}
\begin{equation}
 \sum_{k=3}^n \frac{1}{(k-1)(k-2)}H_{n-k} = \frac{n-3}{n-1} H_{n+1} - \frac{(n-3)(3n^2-1)}{(n-1)^2n(n+1)} 
\end{equation}
\begin{equation}
 \sum_{k=1}^n \frac{1}{k(k+1)(k+2)}H_{n-k} = \frac{n^2+3n-2}{4(n+1)(n+2)} H_{n+1} - \frac{3n-1}{4(n+1)^2} 
\end{equation}
\begin{equation}
 \sum_{k=2}^n \frac{1}{(k+1)k(k-1)}H_{n-k} = \frac{n^2+n-4}{4n(n+1)} H_{n+1} - \frac{4n^3+3n^2-9n-4}{4n^2(n+1)^2} 
\end{equation}
\begin{equation}
 \sum_{k=3}^n \frac{1}{k(k-1)(k-2)}H_{n-k} = \frac{n^2-n-4}{4n(n-1)} H_{n+1} - \frac{(n^2-2n-1)(5n^2+3n-4)}{4(n-1)^2n^2(n+1)} 
\end{equation}
\begin{equation}
\begin{split}
 \sum_{k=2}^n \frac{1}{(k+2)(k+1)k(k-1)}H_{n-k} & = \frac{n^3+3n^2+2n-12}{18n(n+1)(n+2)} H_{n+1} \\
 & - \frac{7n^3+18n^2-25n-12}{36n^2(n+1)^2} 
\end{split}
\end{equation}
\begin{equation}
\begin{split}
 \sum_{k=3}^n \frac{1}{(k+1)k(k-1)(k-2)}H_{n-k} & = \frac{n^3-n-12}{18(n-1)n(n+1)} H_{n+1} \\
 & - \frac{9n^5+9n^4-41n^3-81n^2+32n+24}{36(n-1)^2n^2(n+1)^2} 
\end{split}
\end{equation}

\section{Computer Program}

The Mathematica$^{\textregistered}$ \cite{W03} program used to compute the expressions
is given below.

\begin{alltt}
HarmNumPlus[p_,m_]:=HarmonicNumber[n+1,m]+Sum[1/(n+k)^m,\{k,2,p\}]
HarmNumMinus[p_,m_]:=HarmonicNumber[n+1,m]-Sum[1/(n-k+1)^m,\{k,0,p-1\}]
HarmSumPPos[p_,d_]:=Simplify[
 HarmNumPlus[p,1](HarmonicNumber[n+1]+HarmonicNumber[p-1])
 -1/2((HarmonicNumber[n+1]+HarmonicNumber[p-1])^2
 +HarmonicNumber[n+1,2]+HarmonicNumber[p-1,2])
 -Sum[1/(n+k+2)HarmonicNumber[k],\{k,0,p-2\}]
 -Sum[1/(k+p)HarmonicNumber[k],\{k,0,d-1\}]]
HarmSumPNeg[p_,d_]:=Simplify[
 1/2((HarmNumMinus[p,1]+HarmonicNumber[p])^2
 +HarmNumMinus[p,2]+HarmonicNumber[p,2])
 -HarmonicNumber[n+1](HarmonicNumber[p]+1/(n-p+1))
 +Sum[1/(n-p+k+2)HarmonicNumber[k],\{k,0,p-1\}]
 -Sum[1/(k-p)HarmonicNumber[k],\{k,p+1,d-1\}]]
HarmSumP[p_,d_]:=If[p<=0,HarmSumPNeg[-p,d],HarmSumPPos[p,d]]
HarmSumQPos[p_,d_]:=Simplify[
 HarmonicNumber[n+1](HarmNumPlus[p,1]-HarmonicNumber[p-1])
 -1/2(HarmonicNumber[n+1]^2-HarmNumPlus[p,1]^2
 +HarmonicNumber[n+1,2]+HarmNumPlus[p,2])
 -Sum[1/(n+k+2)HarmonicNumber[k],\{k,0,p-2\}]
 -Sum[1/(k+p)HarmNumMinus[k+1,1],\{k,0,d-1\}]]
HarmSumQNeg[p_,d_]:=Simplify[
 HarmNumMinus[p+1,1]^2-HarmNumMinus[p+1,2]
 -Sum[1/(k-p)HarmNumMinus[k+1,1],\{k,p+1,d-1\}]]
HarmSumQ[p_,d_]:=If[p<=0,HarmSumQNeg[-p,d],HarmSumQPos[p,d]]
HarmTable[m_]:=Table[HarmonicNumber[n+1,i],\{i,m\}]
HarmSumPQ[s_Integer,f_]:=Module[\{d,u,t=HarmTable[2]\},
 d=If[s<=0,-s+1,0];u=Factor[CoefficientArrays[f[s,d],t]];
 u[[1]]+Dot[u[[2]],t]+Dot[Dot[u[[3]],t],t]]
HarmSumPQ[s_,f_]:=Module[\{facs,d,u,funs=0,l=Length[s],t=HarmTable[2]\},
 facs=Table[0,\{l\}];facs[[1]]=1;Do[Do[facs[[j]]/=(s[[i]]-s[[j]]);
 facs[[i]]-=facs[[j]],\{j,1,i-1\}],\{i,2,l\}];
 d=Min[s];d=If[d<=0,-d+1,0];Do[funs+=facs[[i]]f[s[[i]],d],\{i,1,l\}];
 u=Factor[CoefficientArrays[funs,t]];u[[1]]+Dot[u[[2]],t]]
HarmonicSumP[s_]:=If[Length[s]==1,HarmSumPQ[s[[1]],HarmSumP],
 HarmSumPQ[s,HarmSumP]]
HarmonicSumQ[s_]:=If[Length[s]==1,HarmSumPQ[s[[1]],HarmSumQ],
 HarmSumPQ[s,HarmSumQ]]

(* Compute some examples *)
HarmonicSumP[3]//TraditionalForm
HarmonicSumP[\{2,1,0,-1\}]//TraditionalForm
HarmonicSumQ[\{-2\}]//TraditionalForm
HarmonicSumQ[\{0,-1,-2\}]//TraditionalForm
\end{alltt}

\pdfbookmark[0]{References}{}

\end{document}